\documentclass{amsart}
\usepackage{amsmath,amssymb,enumerate,amsthm}
\allowdisplaybreaks[2]

\makeatletter
    
    \@addtoreset{equation}{section}
  \makeatother

\newcommand{\C}{{\mathbb C}}

\newcommand{\N}{{\mathbb N}}
\newcommand{\Z}{{\mathbb Z}}
\newcommand{\R}{{\mathbb R}}
\newcommand{\mO}{{\mathcal O}}
\newcommand{\mN}{{\mathcal N}}

\newcommand{\Ga}{\Gamma}

\newcommand{\la}{\lambda}
\newcommand{\al}{\alpha}
\newcommand{\be}{\beta}
\newcommand{\ga}{\gamma}
\newcommand{\de}{\delta}
\newcommand{\rh}{\rho}
\newcommand{\si}{\sigma}
\newcommand{\om}{\omega}
\newcommand{\vph}{\varphi}

\newcommand{\ggen}{( X^\omega )_{g\textrm{-gen}}}
\newcommand{\gfix}{( X^\omega )_g}
\newcommand{\Ggen}{( X^\omega )_{G\textrm{-gen}}}
\newcommand{\ggeno}{( X^\omega )_{g_2\textrm{-gen}}}
\newcommand{\ggent}{( X^\omega )_{g_1\textrm{-gen}}}
\newcommand{\hxfix}{( X^\omega )_{h(g, x)}}
\newcommand{\hxgen}{( X^\omega )_{h(g, x)\textrm{-gen}}}
\newcommand{\hugen}{( X^\omega )_{h(g, u)\textrm{-gen}}}
\newcommand{\fgen}{( X^\omega )_{f\textrm{-gen}}}
\newcommand{\ffix}{( X^\omega )_f}
\newcommand{\AL}{C$^*$-algebra \ }

\newtheorem{thm}{Theorem}[section]
\newtheorem{prop}[thm]{Proposition}
\newtheorem{cor}[thm]{Corollary}
\newtheorem{lem}[thm]{Lemma}

\theoremstyle{definition}
\newtheorem{defn}[thm]{Definition}
\newtheorem{exa}[thm]{Example}
\newtheorem{rem}[thm]{Remark}
\newtheorem{note}[thm]{Notation}
\newtheorem{mainthm}{}

\title{A von Neumann algebraic approach to self-similar group actions}

\author{Keisuke Yoshida}
\email{kskyhuni@math.sci.hokudai.ac.jp}
\address{Graduate School of Science, Hokkaido University, \mbox{060-0810} Sapporo, Japan}
\date{\today}

\begin{document}
\begin{abstract} 
We study some relations between self-similar group actions and operator algebras.
We see that $\mu(\Ggen)=1$ or $\mu(\Ggen)=0$ where $\mu$ denotes the Bernoulli measure and $\Ggen$ the set of $G$-generic point.
In the case $\mu(\Ggen)=1$,
we get a unique KMS state for the canonical gauge action on the Cuntz--Pimsner algebra constructed from a self-similar group action by Nekrashevych.
Moreover,
if $\mu(\Ggen)=1$,
there exists a unique tracial state on the gauge invariant subalgebra of the Cuntz--Pimsner algebra.
We also consider the GNS representation of the unique KMS state and compute the type of the associated von Neumann algebra.
\end{abstract}

\maketitle

\section{Introduction}
The relations between topological dynamical systems and C$^*$-dynamical systems are often studied in theory of operator algebras.
For instance,
in \cite{KW},
Cuntz--Pimsner algebras were constructed from complex dynamical systems on the Riemann sphere of rational functions by Kajiwara and Watatani.   
One of other (but closely related) examples was suggested in \cite{Nek, Nek2}.
In these papers,
Nekrashevych constructed Cuntz--Pimsner algebras from self-similar group actions.
Self-similar group actions are kinds of actions on the Cantor space $X^\om$ of unilateral infinite words over a finite alphabet $X$.
To mention the definition,
we prepare the unilateral shift $T_x$ for $x \in X$ on $X^\om$ given by $w \mapsto xw$.
A faithful action of a group $G$ on $X^\om$ is said to be \textit{self-similar} if for any $g \in G$ and $x \in X$ there exist $h \in G$ and $v \in X$ such that $gT_x = T_vh$.
If an action is self-similar,
then for any $g \in G$,
$n \in \N$ and $u \in X^n$
there exist $h \in G$ and $v \in X^n$ such that $gT_u = T_vh$.
These $h$ and $v$ are determined uniquely by $g$ and $u$ so we write $h=h(g, u)$ and $v=v(g, u)$.
Considering a rooted tree $X^*$ of finite words over $X$,
we see the reason why the term ``self-similar'' is used. 
From the above formula,
we can identify the action of $g \in G$ on the sub-tree $uX^*$ with the action of $h(g, u)$ on $X^*$.
This self-similar aspect naturally appears on some fractal dynamical systems. 
For example, 
we can construct self-similar groups from self-coverings of topological spaces.
In these examples,
self-similar groups appear as iterated monodromy groups of finite coverings.
Hence we can get self-similar group actions from rational functions.
Rational functions induce Cuntz--Pimsner algebras through two ways.
One is studied in \cite{KW} and another one is in \cite{Nek2} using iterated monodromy groups.
Two algebras are known to be isomorphic in some cases (see \cite{Nek2}) but it is not studied yet when von Neumann algebras constructed from two ways coincide. 
It is one of the purposes of this paper to consider this problem.   
In \cite{IK},
it was studied that the Lyubich measure gives a unique KMS state for the canonical gauge action on a Cuntz--Pimsner algebra of a rational function.
The types of associated von Neumann algebras were also determined in \cite{IK}. 
In this paper,
we will study self-similar group actions from the similar point of view to \cite{IK}.

To see details we again visit the Cuntz--Pimsner algebras of self-similar action.
It is the main idea in the construction of the Cuntz--Pimsner algebras from self-similar group actions to regard $T_x$ as isometry (we write $S_x$ respecting generating isometries in Cuntz algebras) and $g \in G$ as unitary.
The calculation rules should be $gS_x = S_vh$ where $v \in X$ and $h \in G$ are given by
self-similarity. 
We get several Cuntz--Pimsner algebras having the above aspect from self-similar group actions.   
The largest one is the universal C$^*$-algebra generated by $\{S_x\}_{x \in X}$ and $G$ satisfying the above calculation rules,
Cuntz relations
and all relations in $G$.
In this paper,
we use the symbol $\mO_{G_{\max}}$ for the largest Cuntz--Pimsner algebra from self-similar group action of $G$.
For the smallest one,
Nekrashevych has used a topological property called $G$-genericity.
A point $w \in X^\om$ is said to be \textit{$G$-generic} if for any $g \in G$ either $g(w) \neq w$ or $g$ fixes $w$ together with a open neighborhood of $w$.
The set of $G$-generic points is denoted by $\Ggen$ in this paper.
Nekrashevych has considered the norm on $\C G$ given by the permutation representation on $L^2$-functions on the $G$-orbit of a $G$-generic point.  
Actually,
this norm does not depend on the choice of $G$-generic points.
Let $A_{G_{\min}}$ be the completion of $\C G$ with respect to the norm. 
The smallest Cuntz--Pimsner algebra is the one from a Hilbert $A_{G_{\min}}$-bimodule. 
Other algebras from self-similar group actions have been also studied.
Using groupoid theories,
we can construct such algebras. 
Indeed, we can get groupoids from self-similar actions.
The full groupoid C$^*$-algebra is isomorphic to $\mO_{G_{\max}}$ for any self-similar action of $G$.
The reduced one might not be isomorphic to $\mO_{G_{\max}}$ or $\mO_{G_{\min}}$ in general but they are known to be isomorphic in some cases.
If $G$ is amenable,
the groupoid is also amenable (see \cite{Nek2} for Hausdorff groupoid case and \cite{EP} for more general case).
Therefore the full and reduced algebras are coincide.
In some cases with the assumption that $G$ is amenable,
it has been shown that $\mO_{G_{\max}}$ and $\mO_{G_{\min}}$ are isomorphic by proving the simplicity of the reduced groupoid C$^*$-algebra.
See \cite{CE, Nek2}.
It is one of a problems on the C$^*$-algebraic aspect that several Cuntz--Pimsner algebras from a self-similar group action might not be isomorphic especially in non-amenable case.
However,
von Neumann algebras coincides in more general cases.
We do not need the assumption the group is amenable.
This is a good point of the von Neumann algebraic aspect. 
  
This paper is organized as follows.
In the second section,
we will discuss relevance between $G$-generic points and the Bernoulli measure $\mu$ for the following argument on KMS states.
We will see that either $\mu(\Ggen) =1$ or $\mu(\Ggen)=0$ in the second section.
We assume that $\mu(\Ggen)=1$ in later sections.
Indeed,
the class of self-similar group actions with $\mu(\Ggen)=1$ is large enough.
It is easy to see that any free self-similar action satisfies $\mu(\Ggen)=1$.
Another important examples are contracting self-similar group actions.
In \cite{LR, Nek2},
contracting self-similar groups actions were often observed.
The assumption of an action being contracting allows us to look on only finite elements in the group in some sense.
A finite generated group called the Grigorchuk group is one of the famous examples of contracting self-similar groups.  
Actually,
every contracting self-similar group action satisfies $\mu(\Ggen)=1$.

In the third section,
we see that the uniqueness of the special states to get nice von Neumann algebras.
Assuming $\mu(\Ggen)=1$,
we observe that there exist unique KMS states on $\mO_{G_{\max}}$ and $\mO_{G_{\min}}$ for the canonical gauge action.
The existence and uniqueness of the KMS state on $\mO_{G_{\max}}$ is already proved in \cite{LR} for contracting cases and in \cite{BL} for more general cases (the argument in \cite{BL} is not restricted to self-similar group actions).   
In this paper, 
we show the existence of the KMS state on $\mO_{G_{\min}}$.
Furthermore,
we see that there exist unique tracial states on the gauge invariant subalgebras of $\mO_{G_{\max}}$ and $\mO_{G_{\min}}$.
This gives the factority of the gauge invariant subalgebra of the von Neumann algebra which we observe in the last section.
The above states are given by the measure of fixed points of each elements of $G$.

In the last section,
we discuss the von Neumann algebras defined on the GNS space of a unique KMS state on $\mO_{G_{\min}}$.
Von Neumann algebras associated with $\mO_{G_{\min}}$ and $\mO_{G_{\max}}$ coincide though $\mO_{G_{\min}}$ and $\mO_{G_{\max}}$ might not be isomorphic. 
Hence we write $\mO_G^{''}$ for this von Neumann algebra.
At the end of this paper we see that we can compute the type of the von Neumann algebras in nice cases.
In \cite{IK},
it is proved that the von Neumann algebras from rational functions and the Lyubich measure are AFD III$_{\la}$ factors where $0< \la< 1$ is a numbers determined by the degrees of the rational functions.
The following our main result is similar to this one.   
Our first main result (Theorem \ref{type}, Proposition \ref{AFD}) is the following.

\begin{mainthm}
If $\mu(\Ggen)=1$ and $G$ is amenable then $\mO_G^{''}$ is an AFD type III$_{|X|^{-1}}$ factor.
\end{mainthm}

In some cases,
we have the similar result without assuming amenability. 
We prepare a notation for our another main result.
Write $Y_g^n := \{ w \in X^\om \ | \ h(g, w^{(n)})=e\}$ where $w^{(n)} \in X^n$ denotes the first $n$ alphabets of $w$.  
Our another main result (Theorem \ref{main3}) is the following theorem.

\begin{mainthm}
If $\mu(\bigcup_{n \in \N}Y_g ^n) =1$ for any $g \in G$ then $\mO_G^{''}$ is an AFD type III$_{|X|^{-1}}$ factor.
\end{mainthm}

In the proof of the main theorem,
we will see that If $\mu(\bigcup_{n \in \N}Y_g ^n) =1$ then $\mu(\Ggen)=1$.
Moreover,
we will see the assumption of the main theorem holds for a large class of contracting self-similar actions.   

\subsection*{Acknowledgements}
The author appreciates his supervisor,
Reiji Tomatsu, for fruitful discussions and his constant encouragement.
He also expresses his gratitude to Yoshimichi Ueda for giving him the subject of research. 
He would like to thank Yuki Arano for an essential advice for Theorem \ref{type}. 

\section{The Bernoulli measure and $G$-generic points}

We begin with notations which are used in this paper. 
\begin{note}
Let $d$ be a natural number with $d \geq 2$.
Consider a finite set $X$.
In this paper,
$X^*$ denotes the set of finite words over the alphabet $X$.
In other words,
$X^* = \bigcup_{n \in \N} X^n $.
Write $X^\om$ for the set of unilateral infinite words over $X$.
If $w \in X^\om$ and $n \in \N$,
$w^{(n)} \in X^\om $ denotes the first $n$ letters of $w$.
\end{note}

We can identify $X^\N$ with $X^\om$,
and therefore it is equipped with the product topology of discrete topologies on $X$'s.
Thus $X^\om$ is homeomorphic to the Cantor space.
In this paper,
$G$ denotes a countable group.

\begin{defn}(\cite[Definition 2.1]{Nek})
A faithful action of a group $G$ on $X^\om $ is said to be \textit{self-similar}
if
for every $g \in G$ and $x \in X$ there exist $h \in G $ and $v \in X $
such that for any $w \in X^\om $,
\begin{equation}
g(xw) = vh(w).
\label{selfsim}
\end{equation}
%We use a symbol $(G, X)$ for a self-similar group action of a group $G$ over $X^\om$ though self-similar actions might not uniquely determined by $G$ and $X$.
\end{defn}

\begin{rem}
Using the equation (\ref{selfsim}) several times,
we see that for every $n \in \N$,
$g \in G$ and $u \in X^n$ there exist $h \in G $ and $v \in X^n $
such that 
\[
g(uw) = vh(w)
\]
for any $w \in X^\om $.
An easy calculation shows that $h$ and $v$ are uniquely determined by $g$ and $u$ and hence we write $h=h(g, u)$ and $v=v(g, u)$.
\end{rem}

For more details of self-similar actions,
see \cite{Nek2}.

\begin{exa}
Let $X=\{ 0, 1\} $,
and 
$G = (\Z /{2\Z })*(\Z /{2\Z }$). 
Take generators $\al , \be \in G$ with $\al ^2 = \be ^2 = e$ where $e$ is the unit of $G$.
Let $a$ and $b$ be the homeomorphisms on $X^\om $ with defined by: 
\[
a(0w) = 1w , \ a(1w) = 0w, 
\]
\[
 b(0w) = 0a(w), \ b(1w) = 1b(w),
\]
for $w \in X^\om $. 
A map defined by $\al \mapsto a$ and $\be \mapsto b$ is an injective group homomorphism.
Clearly the action is self-similar.  
\end{exa}

\begin{exa}\label{Gri}
Let $X=\{ 0, 1\}$.
Consider homeomorphisms $a, b, c, d$ given by:
\[
a(0w) = 1w , \ a(1w) = 0w, 
\]
\[
 b(0w) = 0a(w), \ b(1w) = 1c(w),
\]
\[
c(0w) = 0a(w) , \ c(1w) = 1d(w), 
\]
\[
 d(0w) = 0w , \ d(1w) = 1b(w),
\]
for $w \in X^\om$ .
Let $G$ be the subgroup of the group of homeomorphisms on $X^\om$ generated by $a, b, c, d$.
The above relations define a self-similar group action.
This group $G$ is called the Grigorchuk group. 
\end{exa}
For more examples,
see \cite{LR, Nek2} and so on.
The countability of $G$ is needed for the following property called $G$-generic.

\begin{defn}(\cite[Definition 4.1]{Nek})
Fix $g \in G$.
An element $w \in X^\om$ is said to be \textit{$g$-generic} if either $g(w) \neq w$ or there exist some neighborhood of $w$ consisting of the fixed points of $g$.
Let $\ggen$ be the set of all $g$-generic points.
Write $\Ggen := \bigcap_{g \in G} \ggen$.
An element $w \in \Ggen$ is said to be \textit{$G$-generic}.
\end{defn} 

Take any $u \in X^*$ and let $T_u$ be the shift operator on $X^\om$ given by $w \mapsto uw$.
Moreover $T_u^*$ denote the partial inverse map of $T_u$ defined on the range of $T_u$.
Write $\langle G, X \rangle := \{ T_{u_1}gT_{u_2}^* \ | \ u_1, u_2 \in X^*, g \in G\}$.

\begin{defn}(\cite[Definition 9.1]{Nek})
Fix $f \in \langle G, X \rangle$.
An element $w \in X^\om$ said to be \textit{$f$-generic} if either $f$ is not defined on $w$ or does not fix $w$ or fix $w$ together with a neighborhood of $w$.
We write $(X^\om)_{f\text{-gen}}$ for the set of $f$-generic points.
Also write $\Ggen ^\text{S} := \bigcap_{f \in \langle G, X \rangle} (X^\om)_{f\text{-gen}}$.
An element in $\Ggen ^\text{S}$ is said to be \textit{strictly $G$-generic}.
\end{defn}

\begin{rem}
For any $f \in \langle G, X \rangle$,
$\fgen$ is a open dense subset of $X^\om$ and therefore $\Ggen ^\text{S}$ is also a dense subset by the countability of $G$.
Let $\ffix$ be the set of fixed points of $f \in \langle G, X \rangle$.
By definition,
$X^\om \backslash \fgen \subset \ffix$.
Note that $\Ggen ^\text{S}$ is a $\langle G, X \rangle$-invariant set.
\end{rem}

Consider the uniform probability measure on the finite set $X$ and let $\mu$ be the product measure of them on $X^\om$.
The measure $\mu$ is often called the Bernoulli measure.
For any $u \in X^*$,
we have $\mu (uX^\om) = | u |^{-1}$ where $| u|$ is the length of $u$ and $uX^\om := \{ w \in X^\om \ | \ w^{(| u | )} = u \}$.

\begin{rem}
Take $f \in \langle G, X \rangle$ and write $f=T_{u_1}gT_{u_2}^*$.
If there exists a fixed point $w$ of $f$ and $|u_1| > |u_2|$ then $w \in u_1X^\om$ and there exists $u \in X^*$ such that $u_1 = u_2u$.
Hence we have a $w_1 \in X^\om$ such that $w=u_1w_1$.
The equation $T_{u_1}gT_{u_2}^*(u_2uw_1)=u_2uw_1$ implies that $g(uw_1)=w_1$.
From the self-similarity,
we have $T_{v(g, u)}h(g, u)w_1=w_1$ and therefore $w_1 \in v(g, u)X^\om$.
Thus there exists $w_2 \in X^\om$ such that $w_1 = v(g, u)w_2$.
Combining $T_{v(g, u)}h(g, u)w_1=w_1$ and $w_1 = v(g, u)w_2$,
we have $h(g, u)v(g, u)w_2=w_2$.
Therefore $w_2 \in v(h(g, u), v(g, u))X^\om$.
Repeating this,
we can uniquely determine $w$ from $u_1$, $u$ and $g$.
Hence the number of fixed points of $f$ is at most one. 
 
Assume that $|u_1| < |u_2|$ and there exists a fixed point $w$ of $f$.
Then there exist $u$ and $w_1$ such that $u_2 = u_1u$,
$w = u_2w_1$ and  $gw_1 = uw_1$.
By the above argument,
the number of fixed points of $f$ is at most one.
Consequently,
the set of fixed points of $f$ is a null set if $|u_1| \neq |u_2|$. 
\end{rem}

The following lemmas are easy but important for our study. 

\begin{lem}\label{gen}
If a countable group $G$ acts self-similarly on $X^\om$,
then the followings are equivalent.
\begin{enumerate}
\renewcommand{\labelenumi}{\textup{(\arabic{enumi})}}
\item
$\mu(\Ggen) =1$.
\item
$\mu(\ggen) = 1$ for any $g \in G$.
\item
$\mu(\Ggen ^\text{S})=1$.
\item
$\mu(\fgen) =1$ for any $f \in \langle G, X \rangle$.
\end{enumerate} 
\end{lem}

\begin{proof}
For the equivalence of (1) and (2) we use an easy measure theoretical argument.
If $\mu(\Ggen)=1$,
then trivially $\mu(\ggen) = 1$ for any $g \in G$.
If $\mu(\ggen)=1$ for any $g \in G$,
then 
\[
\mu(\ggeno \cap \ggent ) = \mu(\ggeno) + \mu(\ggent) - \mu(\ggeno \cup \ggent ) \geq 1
\] 
for any $g_1, g_2 \in G$.
Repeating the same argument several times, 
we have $\mu(\bigcap_{g \in F} \ggen) = 1$ for any finite set $F \subset G$.
Since $G$ is countable,
we get the equivalence of (1) and (2).
Similarly,
(3) and (4) are equivalent.
Since $\Ggen ^\text{S}$ is a subset of $\Ggen$ and therefore (3) implies (1).

At last,
we assume (2) and prove (4). 
Take any $f \in \langle G, X \rangle$.
We show that $\mu(\fgen) =1$.
We write $f=T_{u_1}gT_{u_2}^*$ for some $u_1, u_2 \in X^*$ and $g \in G$.
From the above remark,
we assume that $|u_1| = |u_2|$.
If $u_1 \neq u_2$ then the set of fixed points of $f$ is empty so we assume $u_1=u_2$.
In this case we have $\gfix = u_1\ffix$ and $\gfix \cap \fgen = u_1( \ffix \cap \fgen)$.
The assumption implies that $\mu(\gfix) = \mu(\gfix \cap \ggen)$ and hence $\mu(\ffix) = \mu(\ffix \cap \fgen)$.
Thus (4) holds. 
\end{proof} 

\begin{lem}\label{gen3}
For any $g \in G$,
the following equations hold:
%\begin{itemize}
%\item
\begin{equation}
\lim_{n \to \infty}|X|^{-n} | \{ u \in X^n  \ | \ v(g, u) = u \} | = \mu(\gfix).
\end{equation}
%\item
\begin{equation}
\lim_{n \to \infty}|X|^{-n} | \{ u \in X^n  \ | \ h(g, u) \neq e, v(g, u) =u \} | = \mu(X^\om \backslash \ggen).
\end{equation}
%\end{itemize}
\end{lem}
\begin{proof}
Note that
\[
\bigcup_{\{ u \in X^n  | v(g, u) = u \}} uX^\om = \{ w \in X^\om \ | \ v(g, w^{(n)})=w^{(n)} \}
\]
for any $n \in \N$.
Then 
$\{ w \in X^\om \ | \ v(g, w^{(n)})=w^{(n)} \}$'s are increasing sequence and 
\[
\bigcup_{n \in \N} \{ w \in X^\om \ | \  v(g, w^{(n)})=w^{(n)} \} = \gfix .
\]
Thus the first statement holds.
Similarly,
the second statement does.
\end{proof}

\begin{lem}\label{gen2}
For any self-similar group action of $G$ on $X^\om$ and any $g \in G$,
we have the following equations:
%\begin{description}
%\item
\begin{equation}
\mu(\gfix) = |X|^{-1}\sum_{x \in X} \de_{x, v(g, x)}\mu(\hxfix).
\end{equation}
%\item
\begin{equation}
\mu(\gfix \cap \ggen) = |X|^{-1}\sum_{x \in X} \de_{x, v(g, x)} \mu(\hxfix \cap \hxgen ).
\end{equation}
%\end{description}
where $\de$ is Kronecker's $\de$.
\end{lem}
\begin{proof}
We can easily check that
\[
\gfix = \bigcup_{ \{x \in X | v(g, x)=x \} } x(\hxfix)
\]
and
\[
\gfix \cap \ggen = \bigcup_{ \{x \in X | v(g, x)=x \} } x(\hxfix \cap \hxgen)  
\]
where $xA := \{ xw \ | \ w \in A \}$ for subsets $A$ of $X^\om$. 
We consider the measure on the above equations to finish the proof.
\end{proof}

We observe that Lemma \ref{gen2} gives a KMS condition for gauge actions on the Cuntz--Pimsner algebras associated with self-similar group actions. 
The following theorem shows the dichotomy for self-similar group actions.

\begin{thm}
For any self-similar group action of any countable group $G$ on $X^\om$,
we have either $\mu(\Ggen) =1$ or $\mu(\Ggen) =0$.
\end{thm}

\begin{proof}
If $\mu(\Ggen) \neq 1$,
then $\mu(\ggen)\neq 1$ for some $g \in G$ by Lemma \ref{gen}.
Combining Lemma \ref{gen2} and $X^\om \backslash \ggen = \gfix \backslash (\gfix \cap \ggen)$,
we have 
\[
\mu(X^\om \backslash \ggen) = |X|^{-1}\sum_{x \in X} \de_{x, v(g, x)} \mu(X^\om \backslash \hxgen ).
\]
Note that $( X^\omega )_{e\textrm{-gen}} = X^\om $ where $e$ is the unit of $G$ and we have 
\begin{equation}
\label{ggenc}
\mu(X^\om \backslash \ggen) = |X|^{-1}\sum_{\{ x \in X| h(g, x ) \neq e \} } \de_{x, v(g, x)} \mu(X^\om \backslash \hxgen ).
\end{equation}
Using (\ref{ggenc}) repeatedly,
we have 
\begin{equation}
\label{ggenc2}
\mu(X^\om \backslash \ggen) = |X|^{-n}\sum_{\{ x \in X^n | h(g, u) \neq e\}} \de_{u, v(g, u)} \mu(X^\om \backslash \hugen )
\end{equation}
for any $n \in \N$.
Choose $g_n \in G$ for any $n \in \N$ such that 
\[
\mu(X^\om \backslash ( X^\omega )_{g_n\textrm{-gen}} ) = \max\{ \mu(X^\om \backslash ( X^\omega )_{h(g, u)\textrm{-gen}} ) \ | \ u \in X^n \ \textrm{with} \ v(g, u) = u \ \text{and} \ h(g, u) \neq e \}.
\]
We can take such $\{ g_n \}$ because for any $n \in \N$ there exists at least one $u \in X^n$ such that $v(g, u) = u$ and $h(g, u) \neq e$.
Indeed,
if we have a $n \in \N$ such that for any $u \in X^n$ either $v(g, u) \neq u$ or $h(g, u) =e$ then
$\mu(X^\om \backslash \ggen ) = 0$ by (\ref{ggenc2}),
this is a contradiction.
Using (\ref{ggenc2}),
we have
\[
\mu(X^\om \backslash \ggen) \leq |X|^{-n}\sum_{\{ x \in X^n | h(g, u) \neq e\}} \de_{u, v(g, u)} \mu(( X^\omega )_{g_n\textrm{-gen}} )
\]
for any $n \in \N$.
Taking the limit, 
we obtain
\begin{equation*}
\begin{split}
\mu(X^\om \backslash \ggen) &\leq \lim_{n \to \infty}\mu(X^\om \backslash ( X^\omega )_{g_n\textrm{-gen}} ) |X|^{-n}|\{ x \in X^n \ | \ h(g, u) \neq e \ \textrm{and} \ v(g, u) =u \} | \\
&= \mu(X^\om \backslash \ggen) \lim_{n \to \infty}\mu(X^\om \backslash ( X^\omega )_{g_n\textrm{-gen}}. 
\end{split}
\end{equation*}
We apply Lemma \ref{gen3} above.
Note that $\mu(X^\om \backslash \ggen) \neq 0$ and $\lim_{n \to \infty}\mu(X^\om \backslash ( X^\omega )_{g_n\textrm{-gen}} )$ is at most one.
Then we have $\lim_{n \to \infty}\mu(X^\om \backslash ( X^\omega )_{g_n\textrm{-gen}} ) =1$.
This shows $\mu(\Ggen) = 0$.
\end{proof} 

As a direct corollary of the previous theorem and Lemma \ref{gen},
we have the following one. 
\begin{cor}
For any self-similar action of any countable group $G$ on $X^\om$,
either $\mu(\Ggen ^\text{S})=1$ or  $\mu(\Ggen ^\text{S})=0$.
\end{cor}

In the next section,
we see that a KMS state nicely behaves in case $\mu(\Ggen) =1$.
Many self-similar actions satisfy the condition $\mu(\Ggen) =1$.
To see this,
we recall the definition of contracting.
\begin{defn}(\cite[Definition 2.2]{Nek2})
A self-similar group action of $G$ on $X^\om$ is said to be \textit{contracting} if there exists a finite set $\mN \subset G$ satisfying the following condition:

For any $g \in G$ there exists $n \in \N$ such that we have $h(g, v) \in \mN$ for any $v \in X^*$ whose length is larger than $n$.  

If a self-similar action is contracting,
the smallest finite set of $G$ satisfying the above condition is called the nucleus of $G$.
\end{defn} 

We can actually find the following proposition in the proof of \cite[Theorem 7.3 (3)]{LR} but for reader's convenience we prove here.

\begin{prop}\label{cont}
For any contracting self-similar group action of $G$ on $X^\om$,
we have 
\[
\mu(\Ggen)=1.
\] 
\end{prop}  
\begin{proof}
Take any $g \in G \backslash \{e\}$.
We show that $\mu(X^\om \backslash \ggen) = 0$.
Since the action is contracting,
$\{h(g ,u) \ | \ u \in X^*\}$ is a finite set of $G$.
Hence there exists $m \in \N$ such that for every $u \in X^*$ with $h(g, u) \neq e$ there exists $u' \in X^m$ with $v(h(g, u), u') \neq u'$.
Using inductive argument we see that
\begin{equation}
\label{ind}
|\{ u \in X^{mn} | v(g, u) = u , h(g, u) \neq e \}| \leq (|X|^m - 1)^n
\end{equation}
for any $n \in \N$.
If $n =1$,
then 
\[
| \{ u \in X^m | v(g, u) = u \} | \leq |X|^m -1
\]
by the choice of $m$ (consider the 0 length word $\emptyset \in X^*$ and we have a $u' \in X^m$ with $v(g, u') \neq u'$).
Next we assume that (\ref{ind}) holds for some $n \in \N$.
Note that
\begin{equation*}
%\begin{split}
\{ u \in X^{m(n +1)} | v(g, u) = u, h(g, u) \neq e \} 
=\{ uu' \in X^{m(n +1)} | v(g, uu') = uu', h(g, u) \neq e, h(h(g, u), u') \neq e \}.
%\end{split}
\end{equation*}
For any $ u \in \{ u \in X^{m(n +1)} | v(g, u) = u, h(g, u) \neq e \}$,
we get 
\[
|\{u' \in X^m | v(h(g, u), u') = u', h(h(g, u), u') \neq e \}| \leq (|X|^m - 1)
\]
since there exists an element in $X^m$ which $h(g, u)$ does not fix. 
Hence,
\[
|\{ u \in X^{m(n +1)} | v(g, u) = u, h(g, u) \neq e \}|
\leq (|X|^m -1 )|\{ u \in X^{mn} | v(g, u) = u , h(g, u) \neq e \}| 
\leq (|X|^m -1)^{n +1}
\]
by the inductive assumption.
Consequently, 
We get (\ref{ind}) .
Now we have 
\[
\mu(X^\om \backslash \ggen ) = \lim_{n \to \infty} |X|^{-mn}|\{ u \in X^{mn} | v(g, u) = u , h(g, u) \neq e \}| \leq \lim_{n \to \infty} (|1 - |X|^{-m})^n = 0.
\]
\end{proof}

\section{KMS states on Cuntz--Pimsner algebras associated with self-similar actions}

At first,
we recall Cuntz--Pimsner algebras introduced by Nekrashevych from self-similar group actions in \cite{Nek, Nek2}. 
Recall that we always assume that a countable group $G$ acts self-similarly on $X^\om$.

\begin{defn}(\cite[Definition 3.1]{Nek})
Let $\mO_{G_{\max}}$ be the universal \AL generated by $G$ (we assume that every relation in $G$ is preserved) and $\{ S_x \ | \ x \in X \}$ satisfying the following relations for any $x, y \in X$ and $g \in G$:
\begin{equation}
g^*g = gg^* = 1,
\end{equation}

\begin{equation}
S_x^*S_y = \de_{x, y},
\label{isometry}
\end{equation}

\begin{equation}
\sum_{x \in X}S_xS_x^* = 1,
\label{Cuntz}
\end{equation}

\begin{equation}
\label{similar}
gS_x = S_{v(g, x)}h(g, x),
\end{equation}
where $\de_{x, y}$ denotes the Kronecker delta. 
\end{defn}

For $u=x_1x_2 \cdots x_n \in X^n$,
we write $S_u := S_{x_1}\cdots S_{x_n}$.
Equations (\ref{isometry}) and  (\ref{Cuntz}) imply that $\mO_{G_{\max}}$ contains the Cuntz algebra $\mO_{|X|}$.

\begin{note}
For a strictly $G$-generic point $w \in X^\om$,
$\langle G, X \rangle(w)$ denotes $\langle G, X \rangle$-orbit of $w$,
i.e. $\langle G, X \rangle(w) = \{ f(w) \ | \ f \in \langle G, X \rangle \ \text{such that} \ f \ \text{is defined on} \ w \}$.
Let $L^2(\langle G, X \rangle(w))$ be the set of $L^2$-functions on $\langle G, X \rangle(w)$.
We naturally regard each shift operator $T_x$ as a isometry $S_x$ on $L^2(\langle G, X \rangle(w))$ given by $S_x(\de_{w_1}) :=\de_{xw_1}$ where $w_1 \in \langle G, X \rangle(w)$ and $\{ \de_{w_1} \ |\ w_1 \in \langle G, X \rangle(w) \}$ is a canonical orthonomal basis of $L^2(\langle G, X \rangle(w))$. 
Similarly,
we regard each $g \in G$ as a unitary on $L^2(\langle G, X \rangle(w))$.
We get a representation of $\mO_{G_{\max}}$ on $L^2(\langle G, X \rangle(w))$ thanks to the universality of $\mO_{G_{\max}}$.
We write $\pi_w$ for this representation.
\end{note}

We will recall the following result due to Nekrashevych.

\begin{thm}\textup{(\cite[Theorem 3.3]{Nek2})}
Let $\rh$ be a unital representation of $\mO_{G_{\max}}$ on a Hilbert space.
Then for any $w \in \Ggen ^\text{S}$ and $a \in \mO_{G_{\max}}$,
we have $\| \pi_w(a) \| \leq \| \rh(a) \|$.
\end{thm}

The above theorem implies that the notation $\mO_{G_\text{min}} :=  \pi_w(\mO_{G_{\max}})$
does not depend on the choice of $w \in \Ggen ^\text{S}$.
Let $A_{G_{\min}}$ be the completion of $\C G$ with respect to the C$^*$-norm on $\mO_{G_{\min}}$.
If $G$ is an amenable group,
$\mO_{G_{\min}}$ and $\mO_{G_{\max}}$ are isomorphic in some cases.
See \cite{CE, Nek2} for example.
To understand $\mO_{G_{\min}}$ we consider another algebra from self-similar actions.
The following definition is the almost same one as \cite[Definition 6.1]{Nek}.

\begin{defn}(\cite[Definition 6.1]{Nek})
Let $\mO_{G_{\min}}'$ be the univesal \AL generated by $A_{G_{\min}}$ and $\{ S_x \ | \ x \in X \}$ satisfying equations (\ref{isometry}), (\ref{Cuntz}) and the following relation for any $a \in A_{G_{\min}}$ and $x, y \in X$:
\begin{equation}
aS_x = \sum_{y \in X}S_y\langle y, a\cdot x\rangle,
\label{inner}
\end{equation}
\end{defn} 

The same argument as \cite[Theorem 8.3]{Nek} implies the simplicity of $\mO_{G_{\min}}$.  
\begin{thm}\textup{(\cite[Theorem 8.3]{Nek})}\label{simple}
$\mO_{G_{\min}}'$ is unital, purely infinite and simple.
\end{thm}

Next we observe a Hilbert $A_{G_\text{min}}$-bimodule.
Write $\Phi :=\{ \sum_{\{ x \in X \}} S_xa_x \ | \ a_x \in A_{G_\text{min}} \}$.
We regard this as a right module over the $A_{G_\text{min}}$ with the basis $X$.
Define an $A_{G_\text{min}}$-valued inner product on $\Phi$ by
\[
\langle \sum_{x \in X}S_x a_x \  , 
\sum_{x \in X }S_xb_x \rangle
:= \sum_{x \in X }a_x^* b_x
\]
where $a_x ,
b_x \in A_{G_\text{min}}$ for any $x \in X$. 

We can construct left module structure from self-similarity.
It is given by
\[
g( \sum_{x \in X }S_x a_x ):= \sum_{x \in X}S_{v(g, x)} (h(g, x)a_x)
\]
where $g \in G$ and $a_x \in A_{G_\text{min}}$ for any $x \in X$.
Extending the above definition, we get a $A_{G_\text{min}}$-bimodule.
Indeed let $\phi$ be the *-homomorphism from $A_{G_\text{min}}$ to the set of adjointable maps on $\Phi$ given above then we can check $\phi$ is injective by using the definition of $A_{G_\text{min}}$.
We also use the letter $\Phi$ for this Hilbert $A_{G_\text{min}}$-bimodule.

Actually $\mO_{G_{\min}}'$ is isomorphic to the Cuntz--Pimsner algebra generated from a Hilbert $A_{G_{\min}}$-bimodule $\Phi$ and also isomorphic to $\mO_{G_\text{min}}$.
Indeed,
we can construct universal surjections from $\mO_{G_{\min}}'$ onto $\mO_{G_{\min}}$ and the Cuntz--Pmisner algebra and the surjections are also injective by the simplicity from Theorem \ref{simple}. 
Hence $\mO_{G_\text{min}}'$ is isomorphic to $\mO_{G_\text{min}}$ and a Cuntz--Pimsner algebra generated from a Hilbert $A_{G_{\min}}$-bimodule.

Let $\Ga$ be the canonical gauge action of $\R$ on $\mO_{G_{\min}}$ or $\mO_{G_{\max}}$ given by 
\[
\Ga_t(g):=g  \ \textrm{and}  \ \Ga_t(S_x):=\exp(it)S_x
\]
where $g \in G$,
$t \in \R$ and $x \in X$.
Let $K_\be(\mO_{G_{\min}})$ and $K_\be(\mO_{G_{\max}})$ be the sets of KMS states corresponding to $\Ga$ with inverse temperature $\be > 0$ on $\mO_{G_{\min}}$ and $\mO_{G_{\max}}$,
respectively.
Note that if $\vph \in K_\be(\mO_{G_{\min}})$ (or $K_\be(\mO_{G_{\min}})$),
then 
\[
\vph(S_{u_1} g S_{u_2}^*) = \exp(-\be|u_1|)\phi(g S_{u_2}^*S_{u_1})
\] 
for any $g \in G$ and $u_1, u_2 \in X^*$. 
Combining the above and equation (\ref{Cuntz}),
we have
\[
1 = \vph  (\sum_{x \in X} S_xS_x^* ) = |X|\exp(-\be).
\]
Hence there is no KMS state for $\Ga$ if $\be \neq \log |X|$.
We consider only the case $\be = \log |X|$. 
Let us recall the notation that
$G$ is a countable group which acts self-similarly on $X^\om$ and $e$ is the unit of $G$. 

\begin{defn}[Pre-KMS function]
A positive definite function $\vph$ on $G$ is called a \textit{pre-KMS function}
if
with $\vph(e) = 1$ and
\begin{equation}
\label{pre}
\vph (g) = |X|^{-1}\sum_{x \in X} \de_{x, v(g, x)}\vph(h(g, x))
\end{equation}
for any $g \in G$.
\end{defn}   

\begin{rem}
The restriction of every $\vph \in K_{\log |X|}(\mO_{G_{\min}})$ (or $K_{\log |X|}(\mO_{G_{\max}})$) to $G$ is a pre-KMS function.
Indeed, 
consider the following equation:
\[
 g = g\sum_{x \in X}S_xS_x^* = \sum_{x \in X}S_{v(g, x)}h(g, x)S_x^*
\]
and use the KMS condition of $\vph$,
then we have the equation (\ref{pre}).
\end{rem}

From the above remark,
we should look for pre-KMS functions to find KMS states.
Lemma \ref{gen2} implies that $g \mapsto \mu(\gfix)$ and $g \mapsto \mu(\gfix \cap \ggen)$ are pre-KMS functions on $G$.
If $\mu(\Ggen) = 1$,
then $\mu(\gfix) = \mu(\gfix \cap \ggen)$ by Lemma \ref{gen}.
Actually,
there exists a unique pre-KMS function on $G$ in case $\mu(\Ggen) = 1$. 
The following proposition was already proved in \cite[Proposition 8.3]{BL} but for the reader's convenience we prove here without using more general terminology in \cite{BL}.
 
\begin{prop}\label{main}
If $\mu(\Ggen) = 1$ and $\vph$ is a pre-KMS function on $G$,
then $\vph(g) = \mu(\gfix)$ for any $g \in G$.
\end{prop}
\begin{proof}
For any $n \in \N$ and $g \in G$,
we have 
\[
\vph(g) = |X|^{-n} \ | \ \{ u \in X^n| h(g, u) =e, v(g, u) = u \} | + |X|^{-n} \sum_{ \{ u \in X^n| h(g, u) \neq e \} }\de_{u, v(g, u)}\vph(h(g, u))
\]
from the pre-KMS condition.
Any positive definite function extends to a state on full group algebra and therefore $|\vph(g)| \leq 1$ for any $g \in G$. 
Hence,
\[
| (\vph(g) - |X|^{-n} | \{ u \in X^n \ | \ h(g, u) =e, v(g, u) = u \} |) \ | \leq |X|^{-n} | \{ u \in X^n  \ | \ h(g, u) \neq e, v(g, u) = u \} |.
\]
By Lemma \ref{gen3} and taking limit,
we have 
\[
| \vph(g) - \mu(\gfix \cap \ggen ) | \leq \mu(X^\om \backslash \ggen ).
\]  
The right side of the above inequality is $0$ by the assumption and Lemma \ref{gen}.
Hence 
\[
\vph(g) = \mu(\gfix \cap \ggen )
=\mu(\gfix).
\]
\end{proof}
  
Put $\psi_0(g) := \mu(\gfix)$ for $g \in G$.
As Proposition \ref{main} shows, 
$\psi_0$ is the unique pre-KMS function on $G$ in case $\mu(\Ggen)=1$.
Let $\psi$ be a linear functional on the $*$-algebra generated by $\{ S_x \}_{x \in X}$ and $G$ given by 
\[
\psi(S_{u_1}gS_{u_2}^*) = \de_{u_1, u_2} |X|^{-|u_1|}\psi_0(g)
\]
where $g \in G, u_1, u_2 \in X^*$.
We see that $\psi_0$ extends to the unique KMS state.
To show that $\psi_0$ is positive definite,
we need the next lemma.
The lemma might be proven somewhere as a corollary of \cite[Theorem 2]{FM} but for the reader's convenience we give a proof.

\begin{lem}\label{fixed}
Let $Y$ be a Hausdorff space and $\nu$ be a Borel probability measure on $Y$.
Assume that a countable group $G$ acts faithfully on $Y$ (assume that each $g \in G$ is a measure preserving homeomorphism) and $Y_g$ denotes the set of fixed points of $g \in G$.
Then the map $g \mapsto \nu(Y_g)$ is a positive definite function. 
\end{lem}

\begin{proof}
Let $R := \{ (x, y) \in Y^2 \ | \ y=g(x) \ \text{for some} \ g \in G \}$.
Then $R$ defines an equivalence relation on $Y$.
For $g \in G$,
we define $\text{Graph}(g):=\{ (x, g(x) ) \ | \ x \in Y\} \subset R$.
Moreover,
 consider the projection $\pi_l \colon Y^2 \to Y$ given by $\pi_l((x, y)) := x$. 
We denote by $\nu_l$ the left counting measure (see \cite[Theorem 2]{FM}).
For any $g, h \in G$,
we have 
\[
\nu_l(\text{Graph}(g) \cap \text{Graph}(h) ) = \int_Y |\pi_l(x)^{-1} \cap \text{Graph}(g) \cap \text{Graph}(h)| d\nu
=\nu((Y)_{g^{-1}h}).
\]      
Take $f \in \C G$ and write $f = \sum_{g \in F} \al_g g$ where $F$ is a finite subset of $G$.
We construct a simple function $f' := \sum_{g \in F}\al_g 1_\text{Graph(g)}$.
Then
\[
\sum_{g, h \in F} \overline{\al_g}\al_h\nu(Y_{g^{-1}h})
=\int_Y |f'|^2 d\nu_l \geq 0.
\] 
Thus we have finished the proof.  
\end{proof} 

Before the proof of the uniqueness of KMS state,
we recall the following fact. 
\begin{lem}\textup{(\cite[Lemma3.2]{PW})}
Let $Y$ be a finite dimensional full Hilbert C$^*$-bimodule over $A$ and $\ga >0$.
Take a tracial state $\tau$ on $A$ with
\[
\tau(\sum_{y \in Y} \langle y, ay \rangle ) = \ga \tau(a)
\]
for any $a \in A$.
Then $\tau$ extends to a unique KMS state with inverse temperature $\log \ga $ on the Cuntz--Pimsner algebra $\mO_Y$ over $Y$.  
\end{lem}

\begin{thm}\label{main2}
If $\mu(\Ggen)=1$,
the linear functional $\psi$ extends to a unique KMS state $\psi_{\min}$ in $K_{\log|X|}(\mO_{G_{\min}})$. 
\end{thm}

\begin{proof}
Lemma \ref{fixed} guarantees the positivity of $\psi$.
We see the continuity of $\psi$ on $\C G$ with respect to the norm on $A_{G_{\min}}$.
Take an arbitrary element $f \in \C G$.
Write
\[
f = \sum_{g \in G} \al_g g 
\]
where $\al_g =0$ except for finitely many $g$.
Note that 
\[
\psi(g_1^{-1}g_2) = \mu((X^\om)_{g_1^{-1}g_2}) = \int_{X^\om}\langle g_1(\de_w), g_2(\de_w) \rangle d\mu(w).
\]
Each inner product above is defined on $L^2(\langle G, X \rangle (w))$.
Note that $\mu(\Ggen ^\text{S}) = 1$ by Lemma \ref{gen}.
Hence,
from the definition of $A_{G_{\min}}$ we get 
\begin{equation*}
\begin{split}
\psi(f^*f) 
&= \sum_{g_1, g_2 \in G}\overline{\al_{g_1}}\al_{g_2}\int_{X^\om}\langle g_1(\de_w), g_2(\de_w) \rangle d\mu(w)\\  
&= \int_{X^\om}\langle f(\de_w), f(\de_w) \rangle d\mu(w) \\
&= \int_{\Ggen ^\text{S}}\langle f(\de_w), f(\de_w) \rangle d\mu(w) 
\leq \| f \|_{A_{G_{\min}}}^2.
\end{split}   
\end{equation*}
We have finished proving the continuity and hence $\psi$ defines a tracial state on $A_{G_{\min}}$.
Moreover,
the tracial state satisfies the assumption of the previous lemma and therefore we get a KMS state.
It is easy to see that the KMS state is the extension of $\psi$.

By Proposition \ref{main},
no other tracial state on $A_{G_{\min}}$ satisfies the assumption of the previous lemma.
Thus we have proved the uniqueness part.
\end{proof}
 
The above theorem implies that there exists a KMS state on $\mO_{G_{\max}}$ for the canonical gauge action.
We write $\psi_{\min}$ and $\psi_{\max}$ for the KMS states given by $\psi$ on $\mO_{G_{\min}}$ and $\mO_{G_{\max}}$, respectively.
Note that $\psi_{\min}$ and $\psi_{\min}$ define tracial states on the gauge invariant subalgebras of $\mO_{G_{\min}}$ and $\mO_{G_{\max}}$,
respectively.
We see that they are unique ones.

\begin{thm}\label{trace}
If $\mu(\Ggen)=1$,
then there exists a unique tracial state on the gauge invariant subalgebras of $\mO_{G_{\min}}$ and $\mO_{G_{\max}}$.
\end{thm}  
\begin{proof}
Take a tracial state $\vph$ on gauge invariant subalgebra of $\mO_{G_{\min}}$ or $\mO_{G_{\max}}$.
It is sufficient to prove $\vph(S_{u_1}gS_{u_2}^*) = |X|^{-n}\de_{u_1, u_2}\psi(g)$ for any $g \in G$ and $u_1, u_2 \in X^*$ with same length $n$.
Note that $\vph(S_{u_1}gS_{u_2})=0$ if $u_1 \neq u_2$ and $\vph(S_{u_1}S_{u_1}^*) = \psi(S_{u_1}S_{u_1}^*)$ for any $u_1 \in X^*$ and we get
\begin{equation}\label{eqo}
\begin{split}
\vph(S_{u_1}gS_{u_1}^*) - \psi(S_{u_1}gS_{u_1}^*)
&= \sum_{\{u \in X^m | v(g, u)=u, h(g, u) \neq e\} }(\vph - \psi)(S_{u_1}S_{v(g, u)}h(g, u)S_u^*S_{u_1}^*) \\
&=\sum_{\{u \in X^m | v(g, u)=u, h(g, u) \neq e\} }(\vph - \psi)(S_{u_1}S_{v(g, u)}\text{Re}( h(g, u))S_u^*S_{u_1}^*) \\
&+ i\sum_{\{u \in X^m | v(g, u)=u, h(g, u) \neq e\} }\vph(S_{u_1}S_{v(g, u)}\text{Im}(  h(g, u))S_u^*S_{u_1}^*)
\end{split}
\end{equation}
for any $n \in \N$ where $\text{Re}(h(g, u)):= (h(g, u) + h(g, u)^*)/2$ and $\text{Im}(h(g, u)):= (h(g, u) - h(g, u)^*)/2i$.
Since each $h(g, u)$ is a unitary,
two inequalities
$ -1 \leq  \text{Re}(h(g, u)) \leq 1$ and $ -1 \leq  \text{Im}(h(g, u)) \leq 1$ hold and therefore we have   
\begin{equation}\label{eqtw}
\begin{split}
-\vph(S_{u_1}S_{v(g, u)}S_u^*S_{u_1}^*) - \psi(S_{u_1}S_{v(g, u)}h(g, u)S_u^*S_{u_1}^*)  
&\leq (\vph - \psi)(S_{u_1}S_{v(g, u)}\text{Re}( h(g, u))S_u^*S_{u_1}^*) \\
&\leq \vph(S_{u_1}S_{v(g, u)}S_u^*S_{u_1}^*) - \psi(S_{u_1}S_{v(g, u)}h(g, u)S_u^*S_{u_1}^*)
\end{split}
\end{equation}
and
\begin{equation}\label{eqth}
%\begin{split}
-\vph(S_{u_1}S_{v(g, u)}S_u^*S_{u_1}^*)  
\leq \vph(S_{u_1}S_{v(g, u)}\text{Im}( h(g, u))S_u^*S_{u_1}^*) 
\leq \vph(S_{u_1}S_{v(g, u)}S_u^*S_{u_1}^*) 
%\end{split}
\end{equation}
for any $u \in X^m$ with $v(g, u) = u$ and $h(g, u) \neq e$.
Combining equations (\ref{eqo}),
(\ref{eqtw}) and (\ref{eqth}),
we get
\begin{equation*}
\begin{split}
|\vph(S_{u_1}gS_{u_1}^*) - \psi(S_{u_1}gS_{u_1}^*)| 
&\leq \sum_{\{u \in X^m | v(g, u)=u, h(g, u) \neq e\} } (\vph(S_{u_1}S_{v(g, u)}S_u^*S_{u_1}^*) + \psi(S_{u_1}S_{v(g, u)}h(g, u)S_u^*S_{u_1}^*) ) \\
&+ \sum_{\{u \in X^m | v(g, u)=u, h(g, u) \neq e\} }\vph(S_{u_1}S_{v(g, u)}S_u^*S_{u_1}^*) \\
& \leq 3|X|^{-n-m} |\{u \in X^m | v(g, u)=u, h(g, u) \neq e\}|.
\end{split}
\end{equation*}
By the assumption,
we have $\mu(X^\om \backslash \ggen ) =0$ so Lemma \ref{gen3} and the above inequality imply $\vph(S_{u_1}gS_{u_1}^*) = \psi(S_{u_1}gS_{u_1}^*)$. 
\end{proof} 

From the above theorem,
we can also show that the uniqueness part of Theorem \ref{main2} without using Proposition \ref{main}.
We also have the following theorem as a corollary of the above theorem.
However,
the following theorem was already proved in \cite{BL} from arguments on KMS states on Toeplitz type \AL associated to right LCM monoids.
 
\begin{thm}
If $\mu(\Ggen)=1$,
then $\psi$ extends to a unique KMS state $\psi_{\max}$ in $K_{\log|X|}(\mO_{G_{\max}})$.
\end{thm}
 
\section{von Neumann algebraic approaches}
In this section we consider GNS representations of KMS states which we have discussed at the previous section.
$(\pi_{\psi_{\max}}, H_{\psi_{\max}})$ and $(\pi_{\psi_{\min}}, H_{\psi_{\min}})$ denote GNS representations of $\psi_{\max}$ and $\psi_{\min}$ respectively.
Moreover let $\rho$ be the universal quotient map from $\mO_{G_{\max}}$ onto $\mO_{G_{\min}}$. 
If $\mu(\Ggen)=1$,
then $\psi_{\max} = \psi_{\min} \circ \rho$ and therefore 
$\pi_{\psi_{\max}}(\mO_{G_{\max}})^{''}$ is isomorphic to $\pi_{\psi_{\min}}(\mO_{G_{\min}})^{''}$.
We write simply $\mO_G^{''}$ for this von Neumann algebra.
Moreover, 
we consider only $\mO_{G_{\min}}$ and $(\pi_{\psi_{\min}}, H_{\psi_{\min}})$.
Hence we use a simpler symbol $\psi$ instead of $\psi_{\min}$.
From the simplicity of $\mO_{G_{\min}}$,
$\psi$ is a faithful KMS state and therefore it extends to a normal faithful state on $\mO_G^{''}$.
We also write $\psi$ for this extension. 
Moreover,
the state $\psi$ is a KMS state for the extension of $\Ga$ to $\mO_G^{''}$ (we also use $\Ga$ for the extension of $\Ga$).
The uniqueness of $\psi$ in Theorem \ref{main2} implies the factority of $\mO_G^{''}$.
We see that $\mO_G^{''}$ is an AFD III$_{|X|^{-1}}$ factor in some cases.

First we see that we can compute the type of $\mO_G^{''}$ as a corollary of the uniqueness of tracial states in Theorem \ref{trace}.

\begin{thm}\label{type}
If $\mu(\Ggen)=1$,
then $\mO_G^{''}$ is a type III$_{|X|^{-1}}$ factor.
\end{thm}
  
\begin{proof}
Let $\mO_{G_{\min}}^\Ga$ be the gauge invariant subalgebra of $\mO_{G_{\min}}$.
Then the GNS representation of the restriction of $\psi$ to $\mO_{G_{\min}}^\Ga$ is quasi-equivalent to the restriction of GNS representation $(\pi_\psi, H_\psi)$ to $\mO_{G_{\min}}^\Ga$ by \cite[Lemma 4.1]{Iz}.
By Theorem \ref{trace},
the von Neumann algebra $\pi_{\psi}|_{\mO_{G_{\min}}^\Ga} (\mO_{G_{\min}}^\Ga)^{''}$ on $H_{\psi}|_{\mO_{G_{\min}}^\Ga}$ is a factor and so is $\pi_\psi(\mO_{G_{\min}}^\Ga)^{''}$.
Recalling the uniqueness of the one parameter automorphism group of $\R$ with a certain inverse temperature,
we see that the extension of $\Ga_{-t\log|X|}$ to $\mO_G^{''}$ coincides with the modular automorphism group of $\psi$. 
%Consider an action $\al$ of $\R$ on $\mO_{G_{\min}}$ given by $\al_t := \Ga_{-t\log|X|}$ for any $t \in \R$
%We can easily check that $\psi$ is a KMS state for $\al$ with the inverse temperature $-1$.
%Now we can extend $\al$ to an action $\tilde{\al}$ on $\mO_G^{''}$.
%The expansion $\psi$ of $\psi$ to $\mO_G^{''}$ is also a KMS state for $\tilde{\al}$ with the inverse temperature $-1$.
%Hence we conclude that $\tilde{\al}$ is the modular automorphism group of $\psi$ by the uniqueness of the automorphism which makes a faithful state satisfy the KMS condition with a certain inverse temperature.
We write $\si^{\psi}$ for the modular automorphism group.
Using the periodicity of $\Ga$,
we have the equation 
\[
(\mO_G^{''})^{\si^{\psi}} = (\mO_G^{''})^{\Ga} = \pi_\psi(\mO_{G_{\min}}^\Ga)^{''}
\]   
where $(\mO_G^{''})^{\si^{\psi}}$ is the invariant sublagebra of $\si^{\psi}$.
Thus $(\mO_G^{''})^{\si^{\psi}}$ is a factor and therefore the Connes spectrum of $\si^{\psi}$ coincides with the Arveson spectrum of $\si^{\psi}$ (see \cite[16.1]{S}).
So we can compute the type of $\mO_G^{''}$ using the fact $\si_t^{\psi}$ coincides with the extension of $\Ga_{-t\log|X|}$ and we conclude that $\mO_G^{''}$ is a type III$_{|X|^{-1}}$ factor.
\end{proof}

\begin{prop}\label{AFD}
Assume that $\mu(\Ggen)=1$ and $G$ is amenable.
Then $\mO_G^{''}$ is an AFD type III$_{|X|^{-1}}$ factor.
\end{prop}
\begin{proof}
We only show that $\mO_G^{''}$ is AFD.
If $G$ is amenable then $\mO_{G_{\max}}$ is nuclear by \cite[Corollary 10.16]{EP} and so is $\mO_{G_{\min}}$ which is a quotient of $\mO_{G_{\max}}$.
Now $\pi_\psi$ gives a representation and therefore $\pi_\psi(\mO_{G_{\min}})'$ is semidiscrete.
Thus,
$\mO_G^{''}$ is AFD by \cite[Corollary 3.8.6]{BO}.  
\end{proof}

%By the above proposition,
%we have completely concluded the structure of associated von Neumann algebras in case $\mu(\Ggen)=1$ and $G$ is amenable.
In some cases,
we also show that $\mO_G^{''}$ is an AFD III$_{|X|^{-1}}$ without assuming the amenability.  
Write $Y_g^n := \{ w \in X^\om \ | \ h(g, w^{(n)})=e \}$ for any $g \in G$ and $n \in \N$.
Then $Y_g^n$ is a increasing sequence of open sets and hence $\mu(\bigcup_{n \in \N}Y_g ^n)$ converges.
If it converges to $1$,
we can compute the type of $\mO_G^{''}$. 

\begin{thm}\label{main3}
If $\mu(\bigcup_{n \in \N}Y_g ^n) =1$ for any $g \in G$, 
then $\mO_G^{''}$ is an AFD type III$_{|X|^{-1}}$factor.
\end{thm}
\begin{proof}
Take an arbitrary $g \in G$.
By definition,
we have $(\bigcup_{n \in \N}Y_g ^n) \cap \gfix = \ggen \cap \gfix$ and therefore 
the assumption implies
\[
\mu(\gfix) = \mu((\bigcup_{n \in \N}Y_g ^n) \cap \gfix) = \mu(\ggen \cap \gfix).
\]
Hence $\mu(\Ggen)=1$ by Lemma \ref{gen}.
We show that $G$ is contained in the strong operator closure of $\text{span} \{S_{u_1}S_{u_2}^* \ | \ u_1, u_2 \in X^* \} $.
Take $g \in G$.
Then 
\[
g = \sum_{u \in X^n} S_{v(g, u)}h(g, u)S_u^*
=\sum_{\{u \in X^n| v(g, u) =e\} }S_{v(g, u)}S_u^* + \sum_{ \{u \in X^n | v(g, u) \neq e \} } S_{v(g, u)}h(g, u)S_u^*.
\]
Consider the following sequence 
\[
a_n := \sum_{\{u \in X^n| v(g, u) =e\} }S_v{(g, u)}S_u^*  \in \text{span} \{S_{u_1}S_{u_2}^* \ | \ u_1, u_2 \in X^* \}.
\]
Note that $u \mapsto v(g, u) $ is a bijective map from $X^n$ onto $X^n$ for any $g \in G$ since there exists an inverse map from $g^{-1}$.
Then we compute
\begin{equation*}
\begin{split}
\psi((g - a_n)^*(g - a_n)) &= \sum_{\{u \in X^n| v(g, u)  \neq e\} }\psi(S_uS_u^*) \\
&= |X|^{-n} |\{u \in X^n \ | \ v(g, u)  \neq e\} |
\to 0  \ (n \to \infty).
\end{split}
\end{equation*}
Note that $\{g - a_n \}_{n \in \N}$ is a norm bounded sequence since range projections of $S_{v(g, u)}$'s are mutually orthogonal.
For any $u_1, u_2 \in X^*$ and $g_1 \in G$,
we have 
\begin{equation*}
\begin{split}
\psi( ((g - a_n)S_{u_1}g_1S_{u_2}^*)^*((g - a_n)S_{u_1}g_1S_{u_2}^*) ) 
&= |X|^{|u_2| - |u_1|} \psi((g - a_n)S_{u_1}S_{u_1}^*(g - a_n)^*) \\
&\leq |X|^{|u_2| - |u_1|} \psi((g - a_n)^*(g - a_n))
\end{split}
\end{equation*}     
by the KMS condition.
Thus for any $a \in \text{span} \{S_{u_1}g'S_{u_2}^* \ | \ u_1, u_2 \in X^* \ \text{and} \ g \in G\}$,
\[
\psi( ((g - a_n)a)^* (g - a_n)a) \to 0 \ (n \to \infty).
\] 
Consequently,
$\{a_n\}_{n \in N}$ converges to $g$ in the strong operator topology thanks to norm boundedness of $\{g - a_n \}_{n \in \N}$ and density of $\text{span} \{S_{u_1}g'S_{u_2}^* \ | \ u_1, u_2 \in X^* \}$ in $H_\psi$.
Thus, 
$G \subset \overline{\text{span} \{S_{u_1}S_{u_2}^* \ | \ u_1, u_2 \in X^* \}}$ 
where the closure is the strong operator topological one,
and therefore 
\[
\mO_G^{''} = \overline{\text{span} \{S_{u_1}S_{u_2}^* \ | \ u_1, u_2 \in X^* \}} = \pi_\psi(\mO_{|X|})^{''}.
\]
Now $(\pi_{\psi}|_{\mO_{|X|}}, H_\psi )$ and $(\pi_{\psi|_{\mO_{|X|}}}, H_{\psi|_{\mO_{|X|}}})$ are quasi-equivalent by \cite[Lemma 4.1]{Iz}.
In \cite{Iz},
it is also proved that $\pi_{\psi}|_{\mO_{|X|}}(\mO_{|X|})^{''}$ is an AFD type III$_{|X|^{-1}}$factor.
Hence we finished the proof.
\end{proof}

In the next proposition,
we see that a large class of contracting self-similar group actions satisfies the assumption of Theorem \ref{main3}.

\begin{prop}
Assume that a self-similar group action of $G$ on $X^\om$ is contracting with the nucleus $\mN$.
If for any $g \in \mN$ there exists $u \in X^*$ such that $h(g, u)=e$,
then $\mu(\bigcup_{n \in \N}Y_g ^n) =1$ for any $g \in G$.
\end{prop}
\begin{proof}
It is sufficient to show $\mu(\bigcup_{n \in \N}Y_g ^n) =1$ for any $g \in \mN$.
We use a similar argument to Proposition \ref{cont}.
By assumption,
there exists $m \in \N$ such that for any $g \in G$ there exist $u \in X^m$ with $h(g, u)=e$.
Take any $g \in \mN$.
We show that 
\begin{equation}
\label{ind2}
|\{ u \in X^{mn} \ | \ h(g, u) \neq e \}| \leq (|X|^m -1)^n
\end{equation}
by an inductive argument.
From the choice of $m$,
it is trivial that (\ref{ind2}) holds in case $n = 1$.
For the inductive step,
we assume that (\ref{ind2}) holds for some $n$.
For $u \in X^{m(n+1)}$,
we have a division $u=u_1u_2$ where $u_1 \in X^{mn}$ and $u_2 \in X^m$.
Note that $h(g, u_1) \in \mN$ for any $u_1 \in X^{mn}$ and therefore 
the choice of $m$ implies
\[
|\{ u_2 \in X^m \ | \  h(g, u_1u_2) \neq e\}| \leq |X|^m - 1.
\]
Hence we get
\[
|\{ u \in X^{m(n+1)} \ | \ h(g, u) \neq e \}|
\leq
|\{ u_1 \in X^{mn} \ | \ h(g, u_1) \neq e \}| (|X|^m - 1 ).
\]   
By the inductive assumption,
(\ref{ind2}) holds for $n+1$ and hence we have proved \ref{ind2}.
Thus we get
\[
\mu(X^\om \backslash \ (\bigcup_{n \in \N}Y_g ^n))
= \lim_{n \to \infty}|X|^{-mn}|\{ u \in X^{mn} \ | \ h(g, u) \neq e \}| \leq \lim_{n \to \infty} (1-|X|^{-m} )^n
=0
\]
and therefore we have finished the proof.
\end{proof} 

For the last of this paper,
we compute the type of the von Neumann algebra from the Grigorchuk group.
For the definition of the Grigorchuk group,
see Example \ref{Gri}. 
\begin{exa}
It is known that the Grigorchuk group is contracting and its nucleus is $\{e, b, c, d\}$
(see \cite[Proposition 2.7]{LR}).
We can easily see that the assumption in the above proposition holds by the definition of the Grigorchuk group.
Thus we can conclude that the von Neumann algebra associated with the Grigorchuk group is an AFD factor of type III$_{\frac{1}{2}}$. 
Using the fact that the Grigorchuk group is amenable,
we get the same result. 
\end{exa}

\end{document}